\documentclass{amsart}
\usepackage{amssymb}
\usepackage{graphicx}
\usepackage{amscd}
\usepackage{amsmath}
\usepackage{mathtools}
\DeclarePairedDelimiter\ceil{\lceil}{\rceil}

\newtheorem{theorem}{Theorem}[section]
\theoremstyle{plain}
\newtheorem{lemma}[theorem]{Lemma}
\newtheorem{proposition}[theorem]{Proposition}

\newtheorem{corollary}[theorem]{Corollary}

\theoremstyle{remark}
\newtheorem{definition}[theorem]{Definition}

\newtheorem{examples}[theorem]{Examples}

\newtheorem{remark}[theorem]{Remark}

\numberwithin{equation}{section}

\newcommand{\card}{card}

 \DeclareMathOperator{\sign}{sign}

\def\underset#1\to#2{\mathop{#2}\limits_{#1}{ }}
\def\overset#1\to#2{\mathop{#2}\limits^{#1}{ }}
\begin{document}

\title{\bf Local geometric properties  in quasi-normed Orlicz spaces}


\author[Kami\'nska, {\.Z}yluk]{Anna Kami\'{n}ska and Mariusz {\.Z}yluk}
\address[Kami\'nska]{Department of Mathematical Sciences,
The University of Memphis, TN 38152-3240, U.S.A.}
\email{kaminska@memphis.edu}, 
\address[\.Zyluk]{Department of Mathematical Sciences,
The University of Memphis, TN 38152-3240, U.S.A.}\email{mzyluk@memphis.edu}

\date{}
\maketitle

\begin{abstract} Several local  geometric properties of Orlicz space $L_\phi$ are presented for an increasing Orlicz function $\phi$ which is  not necessarily convex, and thus $L_\phi$ does not need to be a  Banach space. In addition to monotonicity of $\phi$  it is supposed that $\phi(u^{1/p})$ is convex for some $p>0$ which is equivalent to that  its lower Matuszewska-Orlicz index $\alpha_\phi>0$. Such spaces are locally bounded and are equipped with natural quasi-norms.  Therefore many local geometric properties typical for Banach spaces can also be studied in those spaces.  The techniques however have to be different, since duality theory cannot be applied in this case. In this article we present complete criteria, in terms of growth conditions of $\phi$, for $L_\phi$ to have type $0<p\le2$, cotype $q\ge 2$, to be (order)  $p$-convex or $q$-concave, to have an upper $p$-estimate or a lower $q$-estimate, for $0<p,q<\infty$. We provide detailed proofs of most results, avoiding appealing to general not necessary theorems.     
\end{abstract}

We present here a number of results on local geometric properties of Orlicz spaces $L_\phi$, where $\phi:\mathbb{R}_+\to \mathbb{R}_+$, called an Orlicz function, is increasing, $\phi(0)=0$, continuous and $\lim_{u\to\infty} \phi(u) = \infty$. We will assume additionally that the space $L_\phi$ is locally bounded, which means that topology in $L_\phi$ is induced by a quasi-norm.  We shall consider such notions like type, cotype, (order) convexity, concavity and upper, and lower estimates.  In the case of Banach spaces (lattices)  these properties are related, and in fact these relations have been thoroughly studied in the past decades \cite{LT2}.
Although criteria for most of these properties have been known  in the case of normed Orlicz space  $L_\phi$ generated by a convex Orlicz function $\phi$,  they have not been studied thoroughly in the case of quasi-Banach Orlicz spaces. On the other hand in paper \cite{K1998} the author studied Musielak-Orlicz spaces generated by the Orlicz functions (not convex) with parameter, and proved in this generality several results on (order) convexity, concavity, lower and upper estimates, though neither on type nor cotype. Theoretically from those results we can get corollaries for Orlicz spaces $L_\phi$ without parameters. However, both conditions on Orlicz functions with parameter and proofs in Musielak-Orlicz spaces,  are far too complicated comparing with what can be done in just Orlicz spaces. In addition they can not be interpreted instantly in the case of Orlicz functions. 
Therefore it is desirable to provide direct and explicit statements and their proofs of these important properties in spaces $L_\phi$. 

The paper is divided into four sections. The first section contains preliminaries comprising of notations, definitions, some general theorems, and a number of basic results on Orlicz spaces $L_\phi$. The first theorem states that $\alpha_\phi>0$ is necessary and sufficient  for the Minkowski functional to be a quasi-norm in $L_\phi$. It justifies a general assumption made in the entire paper that $\alpha_\phi > 0$.  We also show the relationship among Matuszewska-Orlicz indices $\alpha_\phi$, $\beta_\phi$ and the growth conditions $\Delta_2$, $\Delta^p$ and $\Delta^{*p}$. We  give two regularization theorems on $\phi$, 
which allow us to assume that $\phi(u^{1/p})$ is convex or concave  whenever $\phi\in\Delta^{*p}$ or $\phi\in \Delta^p$ respectively. It is also proved that if $\phi$ does not satisfy condition $\Delta_2$, then $L_\phi$  contains an order isomorphically isometric copy of $\ell_\infty$.  All of these preliminary results are applied in the next sections. 

Section 2 consists of  two main theorems showing that  for any $0<p,q<\infty$, $q$-convexity (resp., $p$-convexity) of  $L_\phi$ is equivalent to lower $q$-estimate (resp. upper $p$-estimate) and this in turn is equivalent to $\phi\in \Delta^q$ (resp., $\phi\in \Delta^{*p}$). In the case of Banach lattices there exist duality among those properties \cite{LT2}. In that case it is possible to prove for instance the conditions on $q$-concavity, and by duality get the conditions on $p$-convexity. However in the case studied here when $L_\phi$ does not need to be a Banach space those theorems must be proved independently. 

In section 3 we present results on type and cotype of $L_\phi$. Again we are able to characterize completely these notions in terms of conditions $\Delta_2$, $\Delta^q$ and $\Delta^{*p}$.  If $1<p\le 2$ then $L_\phi$ has type $p$ if and only if $\phi\in\Delta_2$ and $\phi\in\Delta^{*p}$. For $0<p\le 1$ we show  that $L_\phi$ has type $p$ if and only if $L_\Phi$ is $p$-normable which in turn is equivalent to the fact that $\phi(u^{1/p})$ is convex. In particular for $p=1$ we obtain that $1$-normability, that is the existence of a norm equivalent to the quasi-norm in $L_\phi$, is equivalent to  type $1$. It is not a general fact because there exist spaces with type $1$ that are not normable \cite{KalKam, Kam2018}.

In the last section 4 there are theorems summarizing the conditions and their relationships for type, cotype, convexity, concavity and upper and lower estimates of $L_\phi$ for arbitrary Orlicz function with $\alpha_\phi >0$.  Moreover in the case when $\phi$ is convex there are stated  corollaries on conditions on $B$-convexity and uniform copies in $L_\phi$ of finite dimensional spaces $l_1^n$ and $l_\infty^n$, closely related to nontrivial type and finite cotype, respectively. 

The paper is ended with some examples. We show  that given $0<p\le 2 \le q<\infty$ there exist Orlicz spaces having type $p$ and cotype $q$. Moreover, for any $2\le q<\infty$ there are Orlicz spaces with the upper index $\beta_\phi = q$, which do not have cotype $q$, but they have cotype $q+\epsilon$ for any $\epsilon >0$.

All results on Orlicz spaces remain true in the case of three measure spaces,  a non-atomic measure space with infinite or finite measure, or a discrete measure space identified with the set of natural numbers with counting measure. For each of these measure space, growth conditions of $\phi$ and indices are specified as follows. The growth conditions  $\Delta_2$, $\Delta^q$, $\Delta^{*p}$ and indices $\alpha_\phi$, $\beta_\phi$ defined  for all arguments, large arguments, and small arguments are corresponding to  non-atomic infinite measure, non-atomic finite measure, and discrete counting measure, respectively. We will not repeat this convention while stating the results.

\section{Preliminaries}

Let the symbols $\Bbb R$, $\Bbb R_+$ and $\Bbb N$ stand  for reals, non-negative reals and natural numbers. 
Given a vector space $X$ the functional $x \mapsto
\|x\|$ is called a {\it quasi-norm} if for any $x\in
X$, 
\begin{itemize}
\item[(1)] $\|x\| = 0$
if and only if $x=0$,
\item[(2)] $\|ax\| = |a| \; \|x\|$, $a \in \mathbb{R}$, 
\item[(3)]  there exists $C > 0$ such that for all $x_1, x_2 \in X$,
\[
\|x_1 + x_2 \| \leq C (\|x_1\| + \|x_2\|).
\]
\end{itemize}
  We will
say that $(X, \|\cdot\|)$ is a {\it quasi-Banach
space} if it is complete \cite{KPR}.  Notice that a  quasi-Banach
space is locally bounded \cite{KPR}.  For $0 < p \leq 1$ the
functional $x \mapsto ||x||$ is called a $p$-{\it norm}
if it satisfies the previous conditions (1) and (2) and
condition 
\begin{itemize}
\item[(3')]  for any $x_1, x_2 \in X$, 
\[
\|x_1 +x_2\| \leq (\|x_1\|^p + \|x_2\|^p)^{1/p}.
\]
\end{itemize}
A quasi-Banach space equipped with a $p$-norm is called a $p$-Banach space. 

A quasi-Banach space $(X, \|\cdot\|)$  is called  {\it $p$-normable}   for some $0<p\le 1$ if there exists $C>0$ such that
\[
\|x_1+ \dots + x_n\| \le C(\|x_1\|^p + \dots +\|x_n\|^p)^{1/p}
\]
for all $x_1,\dots, x_n \in X$, $n\in \mathbb{N}$. 
If $X$ is $1$-normable then we say that $X$ is {\it normable}.  For any $p$-normable space $(X,\|\cdot\|)$ there exists a $p$-norm equivalent to $\|\cdot\|$.  In fact
\[
\|x\|_p = \inf\left\{\left(\sum_{i=1}^n \|x_i\|^p \right)^{1/p}:\ x=\sum_{i=1}^n x_i, \, n\in \mathbb{N}\right\}
\]
is a $p$-norm and  $\|x\|_p \le  \|x\| \le C \|x\|_p$ (see \cite{KPR}).

\begin{theorem} {\rm (Aoki-Rolewicz Theorem \cite{KPR, R2})} \label{th:aoki}

 For any quasi-Banach space $(X,\|\cdot\|)$  there exist  $0<p \leq 1$ and a
 $p$-norm $\|\cdot\|_0$ which is equivalent to $\|\cdot\|$. 
\end{theorem}
 
 \begin{remark}\label{rem:aoki}
 
If $(X, \|\cdot\|)$ is a quasi-normed lattice, then it is easy to modify the $p$-norm  above to obtain an equivalent lattice $p$-norm. Indeed, setting $\|x\|_1 = \inf\{\|\,|y|\,\|_0 : |x| \le |y|\}$, the new functional $\|\cdot\|_1$ is a $p$-norm, which preserves the order and is equivalent to $\|\cdot\|$.
 
 \end{remark}

Let $r_n : [0,1]\rightarrow
\Bbb R, n\in\Bbb N$, be {\it Rademacher functions}, that is $r_n(t) =
\sign \,(\sin 2^n\pi t)$, $t\in [0,1]$. A quasi-Banach space $X$ has type $0 <
p\le 2$ if there is a constant $K>0$  such that, for any choice of
finitely many vectors $x_1,\dots ,x_n$ from $X$,
$$
\int_0^1 \left\| \sum_{k=1}^n r_k(t)x_k \right\|dt \le K \left(\sum_{k=1}^n
\|x_k\|^p\right)^{1/p},
$$
and it has cotype $q\ge 2$ if there is a constant $K>0$
such that for any finitely many elements $x_1,\dots ,x_n$ from $X$,
\[
\left(\sum_{k=1}^n \|x_k\|^q\right)^{1/q} \le K \int_0^1\left\|\sum_{k=1}^n r_k(t)
x_k\right\| dt.
\]
Clearly if a quasi-Banach space has type $0<p\le 2$, respectively cotype $q\ge 2$,  then it has type $r$ for any $0<r<p$, respectively cotype $r>q$. For these notions we refer to \cite{LT2} for Banach spaces and to \cite{Kal1981} for quasi-Banach spaces. 

The following  result by Kalton gives a connection between type $1<p\le 2$  and $1$-normability of quasi-Banach spaces.

 \begin{theorem}\label{th:kal1}
 \cite[Theorem 4.1]{Kal1981} Let $1<p\le 2$. If a quasi-Banach space has type $p$ then $X$  is normable. 
\end{theorem}


 A quasi-Banach space
$(X,||\cdot||)$ which in addition is a vector lattice
and $||x|| \leq ||y||$ whenever $|x|\leq |y|$ is called a
{\it quasi-Banach lattice}.  
A quasi-Banach lattice $X=(X,\|\cdot \|)$ is said to be (order) $p$-\textit{convex}, $0
< p < \infty$, respectively (order) $q$-\textit{concave}, $0 < q< \infty,$ if there
are positive constants $C_p$, respectively $D_q$, such that
\[
\left\|\left(\sum_{i=1}^n |x_i|^p \right)^{\frac{1}{p}}\right\| \leq C_p
\left(\sum_{i=1}^n \|x_i\|^p \right)^{\frac{1}{p}}
\]
respectively,
\[
\left(\sum_{i=1}^n \|x_i\|^q \right)^{\frac{1}{q}} \leq D_q
\left\|\left(\sum_{i=1}^n |x_i|^q \right)^{\frac{1}{q}}\right\|
\]
for every choice of vectors $x_1, \ldots, x_n \in X.$ When referring to these notions we skip the word "order", and simply say $p$-convex or $q$-concave. We also say that $X$
satisfies an \textit{upper $p$-estimate}, $0 < p < \infty$, respectively a
\textit{lower $q$-estimate}, $0<q<\infty$, if the inequalities in the  definition of $p$-convexity, respectively $q$-concavity, hold true only for any choice
of disjointly supported elements $x_1, \ldots, x_n$ in $X$
\cite{Kal1981, LT2}. It is well known that if a quasi-Banach space is $p$-convex or has an upper $p$-convexity, then for any $0<r\le p$, $X$ is $r$-convex or has an upper $r$-convexity respectively. For $q$-concavity or lower $q$-estimate the similar property holds in reverse direction \cite{LT2}.  

The next two results on relationships among type, cotype, convexity and concavity hold true in Banach lattices. 

\begin{theorem}\label{th:LT1f18} \cite[Theorem 1.f.18]{LT2}
A Banach lattice $X$ has type $p>1$ if and only if its dual space $X^*$ has cotype $p'$, $1/p + 1/p' =1$, and has a lower $q$-estimate for some $q<\infty$.
\end{theorem}

\begin{theorem}\cite[Corollary 1.f.13]
{LT2}\label{th:1f13} A Banach lattice which has type $p>1$ is $q$-concave for some $q<\infty$.
\end{theorem}

The next result was proved in \cite{LT2}, Theorem 1.d.6  for Banach lattices only, while it is shown here  for quasi-normed lattices.
\begin{theorem}\label{th:Khinchine}
Let $(X,\|\cdot\|)$ be a quasi-Banach lattice. If $X$   is $q$-concave for some $q<\infty$, then there exists $C>0$ such that for every $x_1,\dots,x_n \in X$, $n\in \mathbb{N}$, we have
\[
\int_0^1\left\|\sum_{i=1}^n r_i(t) x_i\right\|dt
\le C \left\|\left(\sum_{i=1}^n |x_i|^2\right)^{1/2}\right\|. 
\]
If $X$   is $p$-convex for some $p>0$, then there exists $C>0$ such that for every $x_1,\dots,x_n \in X$, $n\in \mathbb{N}$, we have
\[
\int_0^1\left\|\sum_{i=1}^n r_i(t) x_i\right\|dt
\ge C \left\|\left(\sum_{i=1}^n |x_i|^2\right)^{1/2}\right\|. 
\]
\end{theorem}
\begin{proof}
Recall first  well known Khintchine's inequality \cite{DJT, LT1}. For every $1\le r < \infty$ there exist positive constants $A_r$ and $B_r$ such that
\[
A_r\left(\sum_{i=1}^n |a_i|^2 \right)^{1/2} \le \left(\int_0^1\left|\sum_{i=1}^n a_i r_i(t)\right|^rdt\right)^{1/r} \le B_r\left(\sum_{i=1}^n |a_i|^2\right)^{1/2}
\]
for every choice of scalars $a_1,\dots, a_n$. We can assume that $0<p<1<q$.
Let $x_1,\dots,x_n \in X$. Then by $q$-concavity of $X$ and Khintchine's inequality,
\begin{align*}
\int_0^1\left\|\sum_{i=1}^n r_i(t) x_i\right\|dt &\le 
\left(\int_0^1\left\|\sum_{i=1}^n r_i(t) x_i \right\|^q dt \right)^{1/q}
=\left(\frac{1}{2^n} \sum_{\theta_j =\pm 1} \left\|\sum_{j=1}^n \theta_j x_j \right\|^q\right)^{1/q}\\
&=\left(\sum_{\theta_j = \pm 1} \left\|\frac{1}{2^{n/q}}\sum_{j=1}^n \theta_j x_j \right\|^q\right)^{1/q}
\le D_q \left\|\left(\sum_{\theta_j=\pm 1}\left|\frac{1}{2^{n/q}} \sum_{j=1}^n \theta_j x_j \right|^q\right)^{1/q}\right\|\\ 
&= D_q \left\| \left(\frac{1}{2^n}\sum_{\theta_j =\pm 1} \left|\sum_{j=1}^n \theta_j x_j \right|^q \right)^{1/q} \right\|
= D_q \left\|\left(\int_0^1 \left|\sum_{i=1}^n r_i(t) x_i \right|^qdt \right)^{1/q}\right\| \\
&\le D_qB_q \left\|\left(\sum_{i=1}^n |x_i|^2\right)^{1/2}\right\|.
\end{align*}
On the other hand applying $p$-convexity of $X$ and Khintchine's inequality,
\begin{align*}
\int_0^1 \left\|\sum_{i=1}^n r_i(t) x_i \right\| dt &\ge \left(\int_0^1 \left\|\sum_{i=1}^n r_i(t) x_i \right\|^pdt \right)^{1/p}
=\left(\sum_{\theta_j=\pm 1} \left\|\frac{1}{2^{n/p}} \sum_{j=1}^n \theta_j x_j \right\|^p\right)^{1/p}\\
&\ge C_p \left\| \left(\sum_{\theta_j=\pm 1} \frac{1}{2^n} \left| \sum_{j=1}^n \theta_j x_j \right|^p \right)^{1/p}\right\|
= C_p \left\|\left(\int_0^1 \left|\sum_{i=1}^n r_i(t) x_i \right|^p dt\right)^{1/p} \right\|\\
&\ge C_p A_p \left\|\left(\sum_{i=1}^n |x_i|^2\right)^{1/2}\right\|.
\end{align*}

\end{proof}

In the sequel $(\Omega,\Sigma,\mu)$ denotes a $\sigma$-finite
measure space. The space
of all (equivalence classes of) $\Sigma$-measurable real
functions defined on $\Omega$ is denoted by $L^0 = L^0(\mu)$. $L^0$
is a lattice with the pointwise order, that is $f\le g$ whenever
$f(t)\le g(t)$ a.e.

In this article we will consider three types of measure spaces.
$(\Omega, \Sigma,\mu)$ will be either nonatomic with $\mu(\Omega)
= \infty$ or nonatomic with $\mu(\Omega) < \infty$, or purely
atomic with counting measure. In the last case we identify
$\Omega$ with $\mathbb{N}$, where $\Sigma$ consists of all subsets of $\Bbb N$ and  $\mu(\{n\}) = 1$ for every $n\in\Bbb N$.

A function $\phi : \Bbb R_+ \rightarrow \Bbb R_+$ is said to be an
{\it Orlicz function} if $\phi(0) = 0$, $\phi$  is increasing (meaning strictly increasing), continuous,  and $\lim_{u \rightarrow \infty} \phi
(u) = \infty$.

Several growth conditions of Orlicz functions will be considered.
These conditions and also equivalence relations between Orlicz
functions will be given in three different versions dependently on
the measure space.  Further, without additional comments (unless
they are necessary), we will always associate the infinite
non-atomic measure with conditions ''for all arguments'', the finite
non-atomic measure with properties ''for large arguments'' and
finally the purely atomic measure with those defined ''for small
arguments''.

In the sequel we will use   the  numbers  called {\it Matuszewska-Orlicz indices},  which
characterize growth conditions of  real valued functions, in particular 
Orlicz functions $\phi$. 
There are three parallel definitions, for all, large or small
arguments.

\begin{definition} Let $\phi$ be an Orlicz function.  The {\it lower
Matuszewska-Orlicz index} $\alpha_\phi$ for all arguments (resp.
large arguments;  small arguments) is defined as follows

\noindent $\alpha_\phi = \sup\{ p\in \Bbb R: $ there exists  $ c>0$
(resp. there exist $c>0$  and $v\ge 0$; there exist $c>0$  and  $v>0$
such that
$$
\phi(au)\ge ca^p \phi(u)
$$
for all $a\ge 1$ and $u\ge 0$ (resp. $u\ge v$; $0 <u\le au\le
v)\}$.

We define {\it upper Matuszewska-Orlicz index} $\beta_\phi$ for all
arguments (resp. large arguments; small arguments) as

\noindent $\beta_\phi = \inf \{ p\in \Bbb R:$  there exists $c>0$
(resp. there exist $c>0$ and $v\ge 0$; there exist $c>0$ and $v>0$
such that
$$
\phi(au)\le ca^p \phi(u)
$$
for all $a\ge 1$ and $u\ge 0$ (resp. $u\ge v$; $0 <u\le au\le
v)\}$.
\end{definition}

The {\it Orlicz space}
generated by an Orlicz function $\phi$ is denoted by $L_\phi$ and
 is defined as the set of all $f\in L^0$ such that
\[
I_\phi(\lambda f) = \int_\Omega \phi(\lambda |f(t)|)d\mu = \int
\phi (\lambda f) < \infty
\]
 for some $\lambda > 0$ dependent on
$f$.  The Minkowski functional of the set $\{f\in L^0: I_\phi(f) \le 1\}$, 
\[
\| f \| = \| f \|_\phi = \inf \{\varepsilon > 0 :
I_\phi({f}/{\varepsilon}) \le 1\}
\]
is finite on $L_\phi$. In the case of counting measure, usually  the Orlicz space is denoted by $l_\phi$ and its
elements are sequences $x = \{x(n)\}$.

The next result provides necessary and sufficient condition for  $\|\cdot\|$ to be a quasi-norm \cite{MatOr}.

\begin{theorem}\label{th:quasinorm}
The Minkowski functional $\|\cdot\|$ is a quasi-norm in $L_\phi$ if and only if 
the upper index $\alpha_\phi > 0$.

 If $\phi$ is a convex function then $\|\cdot\|$ is a  norm in $L_\phi$, called the Luxemburg norm.
 
 \end{theorem}
 
 \begin{proof} We will consider here only the case when the measure is non-atomic and infinite. Then according to our convention we use the index $\alpha_\phi$ for all arguments.
 
 It is clear  that $\|af\|= |a|\|f\|$ for $a\in \mathbb{R}$, and $\|f\|=0$ if and only if $f=0$ a.e..
 If $\alpha_\phi >0$ then there exists $p>0$, such that $\phi(au) \ge C a^p\phi(u)$ for some $C>0$ and all $a\ge1$, $u\ge 0$. Setting then $a>1$ such that $C^{-1}a^{-p} = \frac12$ and $K=a$ we get
 \begin{equation}\label{eq: 001}
 \phi\left(\frac{u}{K}\right) \le \frac12 \phi(u), \ \ \ \ \text{for all} \ \ \ u\ge 0. 
 \end{equation}  
Observe also that in view of monotonicity of $\phi$ we have that 
\[
\phi(\lambda t + (1 - \lambda)s) \le \phi(t) + \phi(s), \ \ \text{for all} \ \ \ t,s \ge 0,\ 0\le \lambda \le1.
\]
Letting then $f,g\in L_\phi$ with $I_\phi\left(\frac{f}{\alpha}\right) \le 1$ and  $I_\phi\left(\frac{g}{\beta}\right) \le 1$ we get
\begin{align*}
I_\phi\left(\frac{f+g}{K(\alpha + \beta)}\right) &\le \int\phi\left(\frac{\alpha}{\alpha+\beta} \frac{|f|}{K\alpha} + \frac{\beta}{\alpha+\beta} \frac{|g|}{K\beta}\right) \\
&\le I_\phi\left(\frac{f}{K\alpha}\right) + I_\phi\left(\frac{g}{K\beta}\right) \\
&\le \frac12 I_\phi\left(\frac{f}{\alpha}\right) + \frac12 I_\phi\left(\frac{g}{\beta}\right) \le 1.  
\end{align*}
 It follows $\|f+g\| \le K(\|f\| +\|g\|)$ for any $f,g \in L_\phi$.  
 
 If $\phi$ is convex then 
 \[
\phi(\lambda t + (1 - \lambda)s) \le \lambda\phi(t) + (1-\lambda)\phi(s), \ \ \text{for all} \ \ \ t,s \ge 0,\ 0\le \lambda \le1.
\]
Then analogously as in general case letting  $I_\phi\left(\frac{f}{\alpha}\right) \le 1$ and $I_\phi\left(\frac{g}{\beta}\right) \le 1$, and applying cnvexity of $\phi$, we get
\begin{align*}
I_\phi\left(\frac{f+g}{\alpha + \beta}\right) \
&\le \frac{\alpha}{\alpha+\beta} I_\phi\left(\frac{f}{\alpha}\right) + \frac{\beta}{\alpha+\beta} I_\phi\left(\frac{g}{\beta}\right) \le 1, 
\end{align*}
 which implies $\|f+g\| \le \|f\| + \|g\|$.

 Let now $\|\cdot\|$ be a quasi-norm. Given a set $A$ with finite measure, clearly 
 \[
 \|\chi_A\| = \frac{1}{\phi^{-1}\left(\frac{1}{\mu(A)}\right)}.
 \]
 By non-atomicity of $\mu$ for any $t>0$ there exist disjoint sets $A, B$ such that $\mu(A) = \mu(B) =t$. By the assumption that $\|\cdot\|$ is a quasi-norm, there exists $M>1$ such that 
 \[
 \frac{\|\chi_A + \chi_B\|}{\|\chi_A\| + \|\chi_B\|} = \frac{\phi^{-1}\left(\frac{1}{t}\right)}{2\phi^{-1}\left(\frac{1}{2t}\right)} \le M, 
 \]
 for all $t>0$. Therefore setting $K=2M$ we get
 \[
 2\phi(t) \le \phi(Kt),
 \]
 for all $t>0$. Taking now $a\ge 1$ there exists $m\in\mathbb{N}$ such that $K^{m-1} \le a < K^m$.  For $p= \frac{\ln 2}{\ln K}$ we have $2^{m-1} = (K^{m-1})^p$. It follows that for all $t\ge 0$ and $a\ge 1$, 
 \[
 \phi(at) \ge \phi(K^{m-1} t) \ge 2^{m-1} \phi(t) = (K^{m-1})^p \phi(t) \ge a^p K^{-p} \phi(t).
 \]
Consequently $\alpha_\phi > 0$, and the proof is completed.  
  
 \end{proof}

 {\bf For the entire paper we assume that $\alpha_\phi > 0$. Then the space $L_\phi$ equipped with the Minkowski functional $\|\cdot\|$ is a quasi-Banach space.}

Recall that Orlicz functions $\phi$ and $\psi$ are {\it equivalent}
for all arguments (resp. large arguments; small arguments) if there
exist  positive constants $K_1, K_2$ (resp.  positive constants $K_1, K_2$ and
$v\ge 0$;  positive constants $K_1, K_2$ and $v>0$) such
that 
\[
K_1^{-1}\psi(K_2^{-1}u)\le \phi(u)\le K_1\psi(K_2u)
\]
 for all $u\ge 0$ (resp.
$u\ge v$; $u\le v$). Two equivalent Orlicz functions define the same
Orlicz spaces with equivalent quasi-norms. In fact it is known more, two Orlicz spaces on non-atomic infinite measure (resp. non-atomic finite measure; purely atomic measure) are equal with equivalent quasi-norms if and only if 
the corresponding Orlicz functions are equivalent for all arguments (resp. large arguments; small arguments) \cite{MatOr, Mus, Lux}.

We start with the most classical growth property of $\phi$,  $\Delta_2$-condition.
We say that an Orlicz function $\phi$ satisfies {\it condition
$\Delta_2$} ($\phi\in \Delta_2$) for all arguments (resp. large arguments; small
arguments) if there exists a positive constant $K$ (resp. there exist a positive constant $K$ and $v\ge 0$; a positive constant $K$ and $v>0$
) such that 
\[
\phi(2u)\le K\phi(u)
\]
 for all $u\ge
0$ (resp. $u\ge v$; $u\le v$). 

The condition $\Delta_2$  is intimately
connected to  more subtle conditions $\Delta^p$ and $\Delta^{*p}$.

\begin{definition}  An Orlicz function $\phi$ satisfies {\it
condition $\Delta^q$} ($\phi\in \Delta^q$), $q>0$, for all arguments (resp. large
arguments; small arguments) if there exists a positive constant
$K$ (resp. there exist a positive constant $K$ and $v\ge 0$; a positive
constant $K$ and $v>0$) such that 
\[
\phi(au)\le K a^q\phi(u)
\]
 for all $a\ge 1$ and $u\ge 0$ (resp. $u\ge v$; $au\le v$).
\end{definition}

An Orlicz function $\phi$ satisfies condition $\Delta^{*p}$ ($\phi\in \Delta^{*p}$), $p> 0$,
 for all arguments (resp. large arguments; small arguments) if
there exists a positive constant $K$ (resp. there exist a positive constant $K$
and $v\ge 0$; a positive constant $K$ and $v>0$)
such that 
\[
\phi(au)\ge K a^p \phi(u)
\]
 for all $a\ge 1$ and $u\ge 0$ (resp. $u\ge v$; $  au\le v$).

\vspace{2 mm}

\begin{remark}\label{rem:delta}
If  $\phi\in \Delta^q$, respectively $\phi\in \Delta^{*p}$,  then   $\phi\in \Delta^{q_1}$, respectively $\phi\in \Delta^{*p_1}$, for any $q_1 > q$, respectively for any $0<p<p_1$. Therefore
\[
\alpha_\phi = \sup\{p : \phi\in \Delta^{*p}\},\ \ \ 
\beta_\phi = \inf\{q : \phi\in\Delta^q\}.
\]
\end{remark}

Given an Orlicz function $\phi$, the function $\phi^*$  is defined as 
\[
\phi^*(u) = \sup_{w\ge0}\{uw - \phi(w)\}.
\]
It is  called the Young conjugate or the compelmentary function, or the Legendre transformation of $\phi$.   Let's point out here that $\phi^*$ may assume
infinite values although $\phi$ itself is finite.
It is easy to
check  that $\phi^* : \Bbb R_+\rightarrow [0,+\infty]$ is convex,
$\phi^*(0) = 0$, $\phi^*$ is left continuous and $\phi^*$ is not
identically equal to infinity.  Moreover, it is well known and easy to check that $\phi^{**}
\le \phi$ and $\phi^{**}$ is the largest convex minorant of $\phi$. Consequently $\phi = \phi^{**}$ whenever $\phi$ is convex \cite[Theorem 1 on p. 175]{IT}.  

Equivalent relations for Orlicz functions possibly
assuming infinite values (like conjugate functions) are defined in the same manner like for the Orlicz finite valued functions.
Observe that in the case when $\phi^*$ satisfies condition
$\Delta_2$ for all and large arguments, it assumes only finite
values, and when this condition is satisfied for small arguments,
then $\phi^*$ may be replaced by an equivalent  Orlicz  convex finite valued function.

Let 
\[
E_\phi= \{f\in L^0:\ I_\phi(\lambda f) < \infty\ \  \text{for every}\ \  \lambda >0\}.
\]  
$E_\phi$ is  a  closed subspace of $L_\phi$ and  is usually called the subspace of finite elements.  It is well known that $E_\phi$ is a closure in $L_\phi$ of
the set of simple functions with finite measure supports. 
Moreover
$L_\phi = E_\phi$ if and only if $\phi$ satisfies condition
$\Delta_2$ \cite[Theorem 8.14]{Mus}. 

If $\phi$ is a convex Orlicz function satisfying condition $\Delta_2$ then  the dual space  $(E_\phi)^*$ is canonically isomorphic
to $L_{\phi^*}$  via integral functionals. In fact 
\[
(E_\phi, \|\cdot\|_\phi)^* \simeq  (L_{\phi^*}, \|\cdot\|_{\phi^*}^0), 
\]
where the symbol $\simeq$ means the spaces are linearly isometric. Here 
\[
\|f\|_{\phi}^0 = \|f\|^0 = \sup\left\{\int_I f\,g : I_{\phi^*}(g) \leq 1\right\}
\]
is the Orlicz norm in $L_\phi$.  It is well known that $\|f\| \leq \|f\|^0 \leq 2\|f\|$ for $f\in L_\phi$. \cite{KR, Lux, Mus}. 

The next result was proved for Musielak-Orlicz spaces in \cite{K1998}.

\begin{proposition}\label{prop:linfty}  If an Orlicz function $\phi$ does not satisfy the corresponding condition
$\Delta_2$, then $L_\phi$ contains an order isomorphically isometric copy of
$\ell_\infty$.
\end{proposition}
\begin{proof}

We shall conduct the proof only in the case of a non-atomic and infinite measure $\mu$. 

Our first observation is that $\phi\in \Delta_2$ if and only if for every $a \ge 1$ there exists $K_a>0$ such that 
\begin{equation}\label{eq:0000}
\phi(au) \le K_a \phi(u)
\end{equation}
 for all $u>0$. In fact for $a>1$ there exists $m \in\mathbb{N}$ such that $2^{m-1} \le a < 2^m$. Setting $q = \ln{K}/ln{2}$ and $K_a = 2^qa^q$, for every $u\ge 0$, 
\[
\phi(au) \le \phi(2^mu) \le K^m\phi(u) = (2^m)^q\phi(u) \le K_a\phi(u).
\]

We claim now that if $\phi$ does not satisfy
$\Delta_2$, then there exists an infinite sequence
$(f_i)$ of functions in $L_\phi$ of disjoint supports
with $I_\phi (f_i) \leq \frac{1}{2^i}$ and
$\|f_i\| = 1$, $i\in \mathbb{N}$.  Indeed, if $\phi$ does not satisfy condition $\Delta_2$ then by (\ref{eq:0000}) there exists  a sequence $(u_n)\subset \mathbb{R}_+$ such that 
\begin{equation}\label{eq:111}
\phi\left(\left(1 + \frac1n\right) u_n\right) \ge 2^n\phi(u_n), \ \ \ n\in\mathbb{N}.
\end{equation}
Without loss of generality we can assume that $(u_n)$ is increasing.
Since $\mu$ is non-atomic and infinite, for every $i\in\mathbb{N}$ there exists an infinite sequence $(A_n^i)$ such that 
\begin{equation*}
\bigcup_{n=1}^\infty A_n^i \cap \bigcup_{n=1}^\infty A_n^j =\emptyset \ \ \ \text{if} \ \  i\ne j
\end{equation*}  and 
\begin{equation}\label{eq:11113}
\phi(u_n) \mu(A_n^i) = \frac{1}{2^n} \ \ \  \text{ for every} \ \ \ i,n\in \mathbb{N}.
\end{equation}
Define
\[
f_i = \sum_{n=i+1}^\infty u_n \chi_{A_n^i}, \ \ \ n\in\mathbb{N}.
\]
Then 
\begin{equation}\label{eq:114}
I_\phi(f_i) = I_\phi\left(\sum_{n=i+1}^\infty u_n \chi_{A_n^i}\right) = \sum_{n=i+1}^\infty \frac{1}{2^n} = \frac{1}{2^i},
\end{equation}
and  for any $\lambda >1$ there exists $n_0$ such that $\lambda > 1 + \frac{1}{n}$ for $n\ge n_0$. Thus  by (\ref{eq:111}) and (\ref{eq:11113}) we get
\[
I_\phi(\lambda f_i) \ge \sum_{n=n_0}^\infty \phi\left(\left(1 + \frac{1}{n}\right) u_n\right)\mu(A_n^i)\ge \sum_{n=n_0}^\infty \frac{2^n \phi(u_n)}{2^n \phi(u_n)} = \infty.
\] 
Consequently
\[
\|f_i\| = 1,\ \ \text{and} \ \ \ f_i\wedge f_j =0, \ \ \text{for} \ \ \ i\ne j.
\]
Hence and in view of (\ref{eq:114}),
\[
I_\phi\left(\sum_{i=1}^\infty f_i \right) = \sum_{i=1}^\infty \frac{1}{2^i} = 1.
\] 
It follows 
\[
\left\|\sum_{n=1}^\infty f_i\right\| =
\|f_i\| = 1, \ \ \ \text{ for} \ \ \ i\in\mathbb{N}. 
\]
 Therefore for  any $x = \{x(n)\} \in {\ell}_\infty$, $n\in\mathbb{N}$,
\[
 |x(n)| = \|x_n f_n\| \le \left\|\sum_{n=1}^\infty x(n) f_n\right\| \ \ \ 
\text{and} \ \ \
\left\|\sum_{n=1}^\infty x(n) f_n\right\| \le \sup_n|x(n)| \left\|\sum_{n=1}^\infty f_n\right\| = \sup_n|x(n)|.
\]
Hence 
\[
\left\|\sum_{n=1}^\infty x(n) f_n\right\| = \sup_n|x(n)| = \|x\|_\infty,
\]
 and thus the closure of the linear span of $(f_n)$ in $L_\phi$ is lattice isomorphically isometric  to $\ell_\infty$.

\end{proof}

\begin{lemma} \label{lem:equiv}
Conditions $\Delta^q$ and $\Delta^{*p}$, $p,q>0$, are
preserved under equivalence relation of Orlicz functions.
\end{lemma}

\begin{proof} We will show this only for $\Delta^{*p}$  when the
equivalence relation and the condition are satisfied for all
arguments.  Let $\phi$ satisfy condition $\Delta^{*p}$ and  $\psi$
be equivalent to $\phi$ that is $K_1^{-1}\psi(K_2^{-1}u)\le \phi(u) \le
K_1\psi(K_2u)$ for all $u\ge 0$ and some $K_1, K_2\ge 1$.  Then for all
$w=K_2^{-1}u\ge 0$ and $a\ge K_2^2$, 
\[
K_1^{-1}\psi(aw) = K_1^{-2} K_1\psi(K_2aK_2^{-2}u)\ge K_1^{-2} 
\phi(aK_2^{-2}u)\ge K_1^{-2}Ka^pK_2^{-2p}\phi(u)\ge K_1^{-3}Ka^pK_2^{-2p}\psi(w).
\]
  If $1\le a\le K_2^2$, then for any $w\ge 0$,
  \[
  \psi(aw)\ge \psi(w) \ge K_2^{-2p}a^p\psi(w).
  \]
  Thus  we showed that $\psi$ satisfies condition
$\Delta^{*p}$.

\end{proof}

The next two results demonstrates possible  regularization of $\phi$ satisfying $\Delta^p$ or $\Delta^{*p}$. The earliest source of these theorems is in \cite{MatOr}. The thorough studies of these can be found in \cite{KMP}.
 
\begin{proposition}\label{prop:deltap}
 Given $q > 0$ the following assertions
are equivalent.
\begin{itemize}
\item[(a)]  An Orlicz function (resp. for $q\ge 1$, convex Orlicz function)  $\phi$ satisfies condition $\Delta^q$.
\item[(b)] There exists an Orlicz function (resp. for $q\ge 1$, convex Orlicz function )  $\psi$ equivalent to
$\phi$ such that for all $u\ge 0$, $a\ge 1$,
$$
\psi(au)\le a^q \psi(u).
$$
\item[(c)] There exists an Orlicz function (resp. for $q\ge 1$, convex Orlicz function)  $\psi$ equivalent to
$\phi$ such that $\psi(u^{1/q})$ is concave.
\end{itemize}
\end{proposition}

\begin{proof} (a) $\Rightarrow$  (c) Let do it first for $\phi$
satisfying $\Delta^q$-condition for large arguments. Define
\[
r(u) =
\begin{cases}
\frac{\phi(v)}{v^q}, &\text{for $0\le u\le v$}\\
\inf_{v\le t\le u}\frac{\phi(t)}{t^q}, &\text{for $u>v$},
\end{cases}
\]
where  $v$ is a constant from condition $\Delta^q$ for large
arguments. Clearly, $r(u)$ is decreasing and continuous and
moreover for every $u\ge v$,
\begin{equation}\label{eq:11}
\frac1K\frac{\phi(u)}{u^q}\le r(u)\le\frac {\phi(u)}{u^q},
\end{equation}
 where $K$ is a constant in condition $\Delta^q$. Set 
 \[
 \psi(u) =\int_0^u r(t)t^{q - 1}dt.
 \]
 Clearly $\psi$ is an Orlicz function.
 
 If $\phi$ is convex and $q\ge 1$, then  the function $\psi$ is also convex since
$r(t)t^{q - 1}$ is increasing. Indeed, it is constant on $(0,v]$, and for $v\le u_1 < u_2$ we have
\[
r(u_1)u_1^{q - 1} = u_1^{q - 1} \inf_{v\le t\le
u_1}\frac{\phi(t)}{t^q}\le u_2^{q - 1}\inf_{v\le t\le
u_1}\frac{\phi(t)}{t^q},
\]
while by (\ref{eq:11}) and in view of that $\phi(u)/u$ is increasing,
\[
r(u_1)u_1^{q - 1}\le \frac{\phi(u_1)}{u_1} = u_2^{q -
1}\inf_{u_1\le t\le u_2} \frac{\phi(u_1)}{t^{q - 1}u_1} \le u_2^{q
- 1}\inf_{u_1\le t\le u_2}\frac{\phi(t)}{t^q}.
\]
Hence
\[
r(u_1)u_1^{q - 1} \le u_ 2^{q - 1}\min \left\{\inf_{v\le t\le
u_1}\frac{\phi(t)}{t^q} , \inf_{u_1\le t\le u_2}\frac{\phi(t)}{t^q}
\right\} = r(u_2)u_2^{q - 1}.
\]

In order to check that $\psi$ is equivalent to $\phi$, observe  by (\ref{eq:11}) that
for $u\ge 2v$, 
\[\psi(u)\ge r(u)\int_v^ut^{q - 1}dt \ge
\frac1q\left(1 - \frac{1}{2^q}\right)r(u)u^q \ge \frac1q\left(1 - \frac{1}{2^q}\right)\frac{1}{K}\phi(u),
\]
 which is half of what is needed. For the other inequality, note  that in view of (\ref{eq:11}) for $u\ge v$, 
\[
\psi(u) \le \int_0^v\frac{\phi(v)}{v^q}t^{q - 1}dt +
\int_v^u\frac{\phi(u)}{u^q}t^{q - 1} dt\le \frac1q(\phi(v) +
\phi(u))\le \frac{2}{q} \phi(u).
\]
Thus $\phi$ is equivalent to $\psi$. Next, an easy change of
variables  yields that $\psi(u^{1/q}) =
\frac1q\int_0^ur(t^{1/q})dt$.  Since $r(t^{1/q})$ is decreasing, we
conclude that the function $\psi(u^{1/q})$ is concave.

Since the case "for all arguments" is  an easy repetition of the
case "for large arguments" setting  $v=0$, what remains to examine
is the case when condition $\Delta^q$ is satisfied near zero.  If
we find an Orlicz function $\rho$ equivalent to $\phi$ and
satisfying condition $\Delta^q$ for all arguments, then the proof
can be done by the first step. Define
\[
\rho(u) =
\begin{cases}
\phi(u), &\text{for $0\le u\le v$}\\
cu^q + (\phi(v) - cv^q), &\text{for $u>v$},
\end{cases}
\]
where  $v$ is the number in condition $\Delta^q$ for small
arguments.   In case of non convex $\phi$, let  $c>0$ be any constant, and  in case of convex $\phi$, set $c= \frac{\phi'(v)}{qv^{q - 1}}$ where $\phi'(v)$ is a
left derivative of $\phi$ at $v$.  Clearly, $\rho$ is increasing  and
equivalent to $\phi$ for small arguments, so it is an Orlicz function.  If $\phi$ is convex then $\rho$ is also convex. 

Now it is enough to show
that $\rho(au)\le Ka^q\rho(u)$ for all $a\ge 1$ , $u\ge 0$ and some
positive constant $K$.  From the assumption of $\Delta^q$ for small arguments this inequality is obvious if $au\le v$. In
the case when $au>u>v$ and $\phi(v) - cv^q > 0$,
\begin{equation}\label{eq:001}
\rho(au) \le a^q [cu^q + (\phi(v) - cv^q)] = a^q \rho(u).
\end{equation}
When $au>u>v$ and $b = -\phi(v) + cv^q > 0$, then
\begin{equation}\label{eq:002}
\frac{\rho(au)}{a^q\rho(u)} \le \frac{ca^qu^q}{a^q(cu^q - b)} =
\frac{cu^q}{cu^q - b} = 1 + \frac{b}{cu^q - b} \le 1 +
\frac{b}{cv^q - b} = 1 + \frac{b}{\phi(v)} = k_0,
\end{equation}
where $k_0$ does not depend on $u$. Finally for $u\le v$ and $au>v$
we get the following  estimation
\[
\rho(au)\le ca^pu^q + \phi(v)\le a^qu^q\left(c + \frac{\phi(v)}{v^q}\right) =
k_1a^qu^q,
\]
with $k_1 = c + \frac{\phi(v)}{v^q}$ is not dependent on $u$. Applying
condition $\Delta^q$ for $a = \frac{v}{u} \ge 1$ and $u\le v$  we
get 
\[
\phi(v) = \phi\left(\frac{v}{u} u\right)\le K\left(\frac{v}{u}\right)^q\phi(u),
\]
which yields $u^q\le \frac{Kv^q}{\phi(v)} \phi(u)$. Hence
\begin{equation}\label{eq:003}
\rho(au)\le k_1a^qu^q\le k_2a^q \rho(u),
\end{equation}
 where $k_2 =\frac{k_1Kv^q}{\phi(v)} > 0$ does not depend on $u$. Combining (\ref{eq:001}), (\ref{eq:002}) and (\ref{eq:003}) we proved that $\rho\in\Delta^q$.\\

(c) $\Rightarrow$ (b) If $\psi(u^{1/q})$ is concave then 
the inequality in (b) is instant. \\

(b) $\Rightarrow$ (a) The inequality in (b) means the $\Delta^q$-condition for all arguments of $\psi$. Since $\phi$ is equivalent with $\psi$,  by  Lemma \ref{lem:equiv}, $\phi$ also satisfies $\Delta^q$-condition. 

\end{proof}

\begin{proposition} 
\label{prop:delta*p}
Given $p > 0$ the following conditions are
equivalent.
\begin{itemize}
\item[$(a)$] The Orlicz (resp. for $p\ge 1$, convex Orlicz) function $\phi$ satisfies condition $\Delta^{*p}$.
\item[$(b)$]There exists an Orlicz (resp. for $p\ge 1$, convex Orlicz) function $\psi$ equivalent to
$\phi$ such that for all $a\ge 1$, $u\ge 0$
$$
\psi(au)\ge a^p\psi(u).
$$
\item[$(c)$] There exists an Orlicz (resp. for $p\ge 1$, convex Orlicz) function $\psi$ equivalent to
$\phi$ such that $\psi(u^{1/p})$ is convex.
\end{itemize}
\end{proposition}

\begin{proof} Since this result is a dual part of Proposition \ref{prop:deltap}, the
proof is analogous. We shall show only the essential construction
part of implication (a) $\Rightarrow$ (c), in the case when the
conditions are satisfied for large arguments.  Define
\[
r(u) =
\begin{cases}
\frac{\phi(v)}{v^{p }}, &\text{for $0\le u\le
v$}\\
\sup_{v\le t\le u}\frac{\phi(t)}{t^p}, &\text{for $u>v$},
\end{cases}
\]
where $v$ is a constant in condition $\Delta^{*p}$.  The function
$r(u)$ is increasing and satisfies the inequality
\begin{equation}\label{eq:115}
\frac{\phi(u)}{u^p}\le r(u)\le \frac1K \frac{\phi(u)}{u^p}
\end{equation}
 with $K$ from condition $\Delta^{*p}$.  Letting 
 \[
 \psi(u) = \int_0^u r(t)t^{p - 1}dt,
 \]
$\psi$ is an Orlicz function. If $\phi$ is convex and $p\ge 1$ then $\psi$ is also convex since $r(w) w^{p-1}$ is increasing.  By changing variables 
\[
\psi(u^{1/p})= \frac1p\int_0^ur(t^{1/p})dt,
\]
and clearly $r(t^{1/p})$ is increasing, so $\psi(u^{1/p})$  is also convex. It remains to  check that $\phi$ and $\psi$ are equivalent.  By (\ref{eq:115}) for $u\ge 2v$,
$$
\psi(u) \ge \int_{u/2}^{u}r(t)t^{p - 1}dt \ge
r\left(\frac{u}{2}\right)\int_{u/2}^u t^{p -1}dt \ge \frac
{\phi\left(\frac{u}{2}\right)}{(\frac{u}{2})^p} \left(\frac1p - \frac1{2^p p}\right)u^p = \phi\left(\frac u2\right)\frac{2^p - 1}{p}.
$$
 For the opposite inequality observe that by condition $\Delta^{*p}$
we have that $\frac1K \frac{\phi(u)}{u^p} \ge \frac{\phi(t)}{t^p}$ for $t\le u$. 
 Letting $u\ge v$,
\[
\psi(u)\le \int_0^v \frac{\phi(v)}{v^{p}} t^{p-1} dt + \frac1K
\int_v^u \frac{\phi(t)}{t^p} t^{p - 1} dt
 \le  \frac{\phi(v)}{p} + \frac{1}{K^2} \int_v^u  \frac{\phi(u)}{u^p} t^{p-1}\, dt \le \frac{1}{pK^2}\phi(u).
\]
 
\end{proof}

\begin{remark} In the case of convex Orlicz function and $p\ge 1$, the method used in Propositions \ref{prop:deltap}, \ref{prop:delta*p} produces function $\psi$ which is not only convex but also smooth.

There exist several other constructions of $\psi$ (eg. \cite{BS, KMP}).  If for instance in the proof of Proposition \ref{prop:delta*p} we  define
 the function $r(u)$  as 
\[
r(u) =
\begin{cases}
\frac{K\phi(v)}{v^p}, &\text{for $0\le u\le v$}\\
\sup_{t\ge u}\frac{\phi(t)}{t^p}, &\text{for $u>v$},
\end{cases}
\]
then $r(u)$ may not be continuous and thus $\psi$ will not be smooth.

\end{remark}

\begin{proposition} \label{prop:dual}
For $1<p<\infty$, a convex Orlicz function $\phi$
satisfies condition $\Delta^{*p}$ if and only if $\phi^*$
satisfies condition $\Delta^{q}$, where $\frac1{p} + \frac1q =
1$.
\end{proposition}

\begin{proof} Since in view of Lemma \ref{lem:equiv}, both conditions $\Delta^{*p}$ and $\Delta^q$
are preserved under equivalence relation,  assuming that $\phi$ satisfies $\Delta^{*p}$, by
Proposition \ref{prop:delta*p} we suppose that $\phi(au)\ge a^p\phi(u)$ for every
$a\ge 1, u\ge 0$. Hence
\[
 \phi^*(au)= \sup_{w>0}\{auw - \phi(w)\} \ge \sup_{w>0}\{auw - a^{-p} \phi(aw)\} = \sup_{w>0}\{ uw - a^{-p}\phi(w)\}=
  a^{-p}\phi^*(a^pu),
 \]
which  implies
that $\phi^*(a^{p}u)\le a^p \phi^*(au)$ or $\phi^*(a^{p - 1}u)\le a^p \phi^*(u)$ for every $a\ge 1 $ and
$u\ge 0$.  Letting $b = a^{p - 1}$, $b\ge 1$, we have $b^{q} = a^p$ and  
\[
\phi^*(bu)\le b^{q}\phi^*(u)
\]
 for every $b\ge 1$ and $u\ge
0$, and so $\phi^*$ satisfies condition $\Delta^{q}$.  The
converse implication can be shown in the similar fashion.

\end{proof}

\begin{proposition}\label{prop:delta2}

{\rm \bf{(I)} } Let $\phi$ be an Orlicz function $\phi$. The following conditions are equivalent.
\begin{itemize}
\item[{\rm(1)}] $\phi$
satisfies condition $\Delta_2$.
\item[{\rm(2)}] $\phi$ satisfies
condition $\Delta^q$ for some $0<q<\infty$.
\item[{\rm(3)}]   $\beta_\phi<\infty$.
\end{itemize}

{\rm \bf{(II)}} Let $\phi$ be a convex Orlicz function. The following conditions are equivalent.
\begin{itemize}
\item[{\rm(1')}] $\phi^*$
satisfies condition $\Delta_2$.
\item[{\rm(2')}] $\phi$ satisfies
condition $\Delta^{*p}$ for some $p>1$.
\item[{\rm(3')}]   $\alpha_\phi>1$.
\end{itemize}

\end{proposition}

\begin{proof}

{\bf (I)}.  The equivalence of (2) and (3)   follows from Remark \ref{rem:delta}.

The implication from (2) to (1)  is
obvious.  

It is enough to show (1) implies (2).  We shall prove this only in the case when the
relations are satisfied  for large arguments.  The inequality
$\phi(2u)\le K\phi(u)$ is valid for all $u\ge v$, where $v\ge 0$
and $K>1$ are some constants.  
Then for any $a \ge 1$ there exists $m \in\mathbb{N}$ such that $2^{m-1} \le a < 2^m$. Setting $q = \ln{K}/ln{2}$, for every $u \ge v$ and any $a \ge 1$,
\[
\phi(au) \le \phi(2^mu) \le K^m\phi(u) = (2^m)^q\phi(u) \le 2^qa^q\phi(u),
\]
which shows that $\phi$ satisfies condition $\Delta^q$. 

{\bf (II)} 
From part (I), $\phi^*$ satisfies $\Delta_2$ if and only if $\phi^*$ satisfies $\Delta^q$ for some $1<q<\infty$. By the assumption that $\phi$ is convex $\phi = \phi^{**}$.  Thus by Proposition \ref{prop:dual}, $\phi^*\in\Delta^q$  is equivalent to 
  $\phi\in\Delta^{*p}$, where $1/p + 1/q =1$. Therefore (1') is equivalent to (2').

Finally (2') is equivalent to (3') by Remark \ref{rem:delta}.

\end{proof}

\section{Convexity, concavity, lower- and upper-estimates}

\begin{theorem}\label{th:concavity}
 Let $0<q<\infty$ and $\phi$ be an Orlicz function. Then the following properties are equivalent.
 \begin{itemize}
 \item[(i)]  The Orlicz space $L_\phi$ is $q$-concave.
 \item[(ii)] $L_\phi$ satisfies a lower $q$-estimate.
 \item[(iii)]  $\phi$ satisfies condition $\Delta^q$.
 \end{itemize}
\end{theorem}

\begin{proof} 
Since for disjointly supported functions $f_1,\dots,f_n \in L_\phi$,
\[
\left\|\sum_{i=1}^n f_i\right\|= \left\|\left(\sum_{i=1}^n |f_i|^q\right)^{\frac{1}{q}}\right\|,
\]
clearly (i) implies (ii). 

Assume (ii) and show (iii).  Let $L_\phi$  satisfy a lower $q$-estimate.

Let us
prove it first when the measure $\mu$  is non-atomic and finite.
Without loss of generality assume that $\phi(1) = 1$ and $\mu(\Omega) = 1$. For arbitrary $u>w\ge 1$, set 
\[
x =\frac1{\phi(u)} \ \ \ \text{and} \ \ \  y = \frac1{\phi(w)}
\]
 and $n = [\frac{y}{x}]$ is  the entire part of $\frac{y}{x}$.  Clearly $0<x<y\le 1$ and $n\le \frac{y}{x} \le
2n$.  Since $nx\le y \le 1$, we are able to choose $n$ disjoint
measurable sets $A_i$ with $\mu A_i = x$ for every $i = 1,\dots,
n$. Now, setting
\[
f_i = \phi^{-1} (1/x) \chi_{A_i} \ \ \  i = 1,\dots n, 
\]
where $\phi^{-1}$ is an inverse function of $\phi$ on $(0,\infty)$,
we obtain that $\| f_i \| = 1$ and
\[
\left\|\sum_{i=1}^n f_i\right\| = \phi^{-1} (1/x) \| \chi_{\bigcup_{i=1}^n
A_i} \|  = \frac {\phi^{-1}(1/x)}{\phi^{-1}(1/(nx))}.
\]

If $D_q$ is a lower $q$-estimate constant then in view of $1/y \le 1/(nx)$,
\[
\frac{u}{w} =\frac {\phi^{-1}(1/x)}{\phi^{-1}(1/y)} \ge \frac
{\phi^{-1}(1/x)}{\phi^{-1}(1/(nx))} = \left\| \sum_{i=1}^n f_i
\right\| \ge D_q \left(\sum_{i=1}^n \|f_i\|^q\right)^{\frac1q}  = D_q n^{\frac1q} \ge k
\left(\frac{y}{x}\right)^{\frac1q} = k \left(\frac{\phi(u)}{\phi(w)}\right)^{\frac1q},
\]
with $k = D_q 2^{-\frac1q}$.  This implies that for every $u>w\ge
1$, 
\[
\phi(u) \le k^{-q} ({u}/{w})^q \phi(w),
\]
and  thus setting $s=w$ and $as=u$ where $a\ge 1$ and $s\ge 1$, 
\[
\phi(as)\le k^{-q} a^q \phi(s),
\]
which means $\Delta^q$ condition of $\phi$ for large arguments. 

When the measure is non-atomic and infinite, then the proof is similar and in fact  simpler.

Let's  sketch the proof in the case of  discrete measure.  Assume also that $\phi(1) = 1$. For any
$0<w<u\le 1$, let $x = \frac1{\phi(u)}$ and $y = \frac1{\phi(w)}$
and $n = [\frac{y}{x}]$.  There exists $m\in \Bbb N$ with $m\le x
< m+1$. Choose $n$ disjoint subsets $N_i$ of $\Bbb N$ with $\card{N_i} = m$ and define
\[
x_{i}(j) =
\begin{cases}
\phi^{-1}(1/x),&\text{for $j\in N_{i}$} \\ 0,
&\text{for  $ j\notin N_{i}$},
\end{cases}
\]
for $i=1,\dots, n$. Then the vectors $x_i\in l_\phi$ have disjoint supports and
$\|x_i\| = \frac{\phi^{-1}(1/x)}{\phi^{-1}(1/m)}$.

By the general assumption $\alpha_\phi > 0$, in view of Proposition \ref{prop:delta*p}    there exists $p>0$ and equivalent Orlicz function $\psi$ to $\phi$ such that $\psi(u^{1/p})$ is convex. Without loss of generality we can assume that $\phi(u^{1/p})$ is convex. It follows that the function ${\phi^{-1}(u^p)}/{u}$ is decreasing. Consequently,
$$
\|x_i\| = \frac{\phi^{-1}(1/x)}{\phi^{-1}(1/m)} \ge
\frac{\phi^{-1}(1/(m+1))}{\phi^{-1}(1/m)} \ge
\left(\frac{m}{m+1}\right)^{1/p}.
$$
By  
\[
\frac{1}{mn} \ge \frac{1}{y}, \ \ \ \ n^{\frac{1}{q}}\ge \left(\frac{1}{2}\right)^{{1/q}}\left(\frac{y}{x}\right)^{1/q},
\]
 and by the lower $q$-estimate of $l_\phi$,  
\begin{align*}
\frac{u}{w} &= \frac{\phi^{-1}(1/x)}{\phi^{-1}(1/y)} \ge \frac{\phi^{-1}(1/x)}{\phi^{-1}(1/(nm))} = \left\|\sum_{i=1}^n x_i\right\| \ge D_q\left(\sum_{i=1}^n \left(\frac{m}{m+1}\right)^{\frac{q}{p}} \right)^{\frac1q}\\
&=  D_q n^{\frac{1}{q}}
\left(\frac{m}{m+1}\right)^{\frac{1}{p}} 
\ge D_q \left(\frac{1}{2}\right)^{\frac1q+\frac1p }  \left (\frac{y}{x}\right)^{1/q} = k \left(\frac{\phi(u)}{\phi(w)}\right)^{\frac1q},
\end{align*}
for every $0 < w < u \le 1$,  where $D_q$  is the constant of lower $q$-estimate and $k = D_q (1/2)^{\frac1q +\frac1p}$ a constant not depending on $x$ and $y$. Similarly as in the first case it implies condition $\Delta^q$ for small arguments.

Below an alternative proof is provided in the case of finite non-atomic measure. Observe  at first that
$\phi$ must satisfy condition $\Delta_2$.  Indeed,  in the opposite
situation in view of Proposition \ref{prop:linfty}, $l_\infty$ is an order isometric copy  in $L_\phi$, and
so $L_\phi$ can not satisfy any lower $q$-estimate. Note that it is easy to show that $l_\infty$ does not have any lower $q$-estimate for any $q<\infty$.  

Assume now that $\phi$ does not satisfy condition $\Delta^q$ for large arguments.  Assume without loss of generality  that
$\phi(1) = 1$. Since $\phi$ does not satisfy $\Delta^q$ for large arguments,  there
exist infinite sequences $\{a_n\}$ and $\{u_n\}$, such that $1\le
u_n \rightarrow \infty$ as $n\rightarrow \infty$, $a_n\ge 1$ and
\begin{equation}\label{eq:1222}
\phi(a_n^{1/q}u_n) \ge 2^n\, a_n\,\phi(u_n),
\end{equation}
for every $n\in\Bbb N$. The sequence $\{a_n\}$ is not bounded.
Indeed,  if for some $M$, $a_n\le M$, then  by $\Delta_2$ condition,
\[
2^na_n\phi(u_n)\le \phi(a_n^{1/q}u_n) \le \phi(M^{1/q}u_n) \le K
\phi(u_n)
\]
for all $n\in\mathbb{N}$ and some $K>0$, which is a contradiction. Therefore we assume that
$\sum_{n=1}^\infty \frac1{a_n} < \infty$, extracting a subsequence
if necessary.  Denoting by $[a_n]$ the entire part of $a_n$ with
$[a_0] = 0$, let
$$
N_n = \left\{ 1 + \sum_{i=0}^{n-1} [a_i],  2 + \sum_{i=0}^{n-1} [a_i],
\dots, [a_n] + \sum_{i=0}^{n-1} [a_i] \right\},
$$
for every $n\in \Bbb N$.  The sequence $\{N_n\}$ is a disjoint
partition of the natural numbers $\Bbb N$, and each $N_n$ contains
exactly $[a_n]$ consecutive natural numbers.  There will be no loss
of generality if we suppose that $\mu (\Omega) = 1$. Observe that by (\ref{eq:1222}),
\[
\sum_{n=1}^\infty \frac{[a_n]}{\phi(a_n^{1/q}u_n)} \le
\sum_{n=1}^\infty \frac1{2^n} = 1.
\] 
Therefore for every $n\in\Bbb
N$ there exists a finite sequence of sets $\{A_{in}\}_{i\in N_n}$
such that all $A_{in}$ are disjoint and
$$
\mu (A_{in}) = \frac1{\phi(a_n^{1/q}u_n)}
$$
for each $i\in N_n$, $n\in \Bbb N$.  Define
$$
g_{in} = u_n\chi_{A_{in}}
$$
for $i\in N_n$ and $n\in \Bbb N$.  For every $n$, $\{g_{in}\}_{i\in
N_n}$ is a finite sequence of functions with the same distributions
and such that  
\[
I_\phi (a_n^{1/q}g_{in}) = \phi (a_n^{1/q}u_n) \mu
(A_{in}) = 1.
\]
 Hence $\| g_{in} \| = 1/a_n^{1/q}$ for every $i\in
N_n , n\in\Bbb N$. It then implies that for sufficiently large $n\in
\Bbb N$,
\begin{equation}\label{eq:1333}
\sum_{i\in N_n} \|g_{in}\|^q = \sum_{i\in N_n} \frac1{a_n} = \frac
{[a_n]}{a_n} \ge \frac12.
\end{equation}
On the other hand since $\phi$ satisfies condition $\Delta_2$ for
large arguments, for every $\lambda > 1$ there exists a constant
$K$ such that $\phi(\lambda u) \le K\phi (u)$ for all $u\ge 1$.
Therefore by (\ref{eq:1222}),
\begin{equation}\label{eq:1444}
I_\phi\left( \lambda \sum_{i\in N_n} g_{in}\right) =
[a_n]\phi(\lambda u_n)\mu(A_{in}) = \frac{[a_n]\phi(\lambda u_n)}{ \phi (a_n^{1/q}u_n)}  \le \frac{[a_n] \phi (\lambda
u_n)}{2^n a_n \phi(u_n)}\le  \frac{K}{2^n},
\end{equation}
where $K$ depends only on $\lambda$. Hence for all $\lambda > 0$,
$I_\phi (\lambda \sum_{i\in N_n} g_{in}) \rightarrow 0$, which
means that $\|\sum_{i\in N_n} g_{in} \| \rightarrow 0$ as
$n\rightarrow \infty$. Thus in view of (\ref{eq:1333}), $L_\phi$ cannot satisfy any lower $q$-estimate.

We shall show now the implication from (iii) to (i). Let $\phi$ satisfy condition $\Delta^q$.  In view of Proposition \ref{prop:deltap} there exists a Young function
$\psi$ equivalent to $\phi$ such that $\psi(u^{1/q})$ is concave.
Since the norms determined by both functions $\psi$ and $\phi$ are
equivalent, we assume not losing generality that $\phi(u^{1/q})$ is
concave. By direct calculations we get

\[
\|f\|_{\phi(u^{1/q})} = \| |f|^{1/q} \|_\phi^q.
\]
 Notice also
that the Luxemburg functional defined by means of  a concave
function (e.g. $\|\cdot  \|_{\phi(u^{1/q})}$) satisfies the reverse  triangle
inequality.  Thus
$$
\left\|\left(\sum_{i=1}^n |f_i|^q\right)^{1/q} \right\|_\phi =\left\| \sum_{i=1}^n |f_i|^q
\right\|_{\phi(u^{1/q})}^{1/q}\ge \left(\sum_{i=1}^n \| \,|f_i|^q
\|_{\phi(u^{1/q})}\right)^{1/q}=\left(\sum_{i=1}^n \|f_i\|_\phi^q\right)^{1/q},
$$
which shows that $L_\phi$ is $q$-concave.
\end{proof}

\begin{theorem}\label{th:convexity}  Let $0<p<\infty$ and $\phi$ be an Orlicz function. Then the following conditions are equivalent.
\begin{itemize}
\item[(i)] The Orlicz space $L_\phi$ is $p$-convex.
\item[(ii)] $L_\phi$ satisfies an upper $p$-estimate. 
\item[(iii)]  $\phi$ satisfies condition $\Delta^{*p}$.
\end{itemize}
\end{theorem}

\begin{proof}
It is clear that (i) yields (ii).

(ii) $\to$ (iii)
Let now  $L_\phi$ satisfy an upper $p$-estimate. We will show (iii) that is $\phi\in\Delta^{*p}$. We are giving the proof only in the case of finite and non-atomic measure. Assume without loss of generality that $\phi(1) = 1$ and $\mu(\Omega) = 1$. For the sake of convenience we will repeat several steps from the proof of the implication (ii) to (iii) of Theorem \ref{th:concavity}.
 
 For arbitrary $u>w\ge 1$, set 
\[
x =\frac1{\phi(u)} \ \ \ \text{and} \ \ \  y = \frac1{\phi(w)}
\]
 and $n = [\frac{y}{x}]$ is  the entire part of $\frac{y}{x}$.  Clearly $0<x<y\le 1$ and $n\le \frac{y}{x} \le
2n$.  Since $nx\le y \le 1$, we are able to choose $n$ disjoint
measurable sets $A_i$ with $\mu A_i = x$ for every $i = 1,\dots,
n$. Now, setting
\[
f_i = \phi^{-1} (1/x) \chi_{A_i} \ \ \  i = 1,\dots n, 
\]
where $\phi^{-1}$ is the inverse function on $(0,\infty)$ of $\phi$,
we obtain that $\| f_i \| = 1$, $f_i\wedge f_j =0$ for $i\ne j$, and
\[
\left\|\sum_{i=1}^n f_i\right\| = \phi^{-1} (1/x) \| \chi_{\bigcup_{i=1}^n
A_i} \|  = \frac {\phi^{-1}(1/x)}{\phi^{-1}(1/(nx))}.
\]

The general assumption $\alpha_\phi >0$ implies that $\phi^{-1}$ satisfies condition $\Delta_2$. Indeed there exist a constant $1>C>0$, $r>0$ such that $\phi(at) \ge Ca^r\phi(t)$ for all $a\ge 1$ and $t\ge 0$. Hence for all $s =\phi(t)\ge 0$ we have $a\phi^{-1}(s) \ge \phi^{-1}(Ca^r s)$. Setting $K = a=(2/C)^r$, we get 
 \[
 \phi^{-1}(2 s) \le (2/C)^r \phi^{-1}(s) = K \phi^{-1}(s).
 \] 
If $C_p$ is an upper $p$-estimate constant then in view of $1/y\ge 1/(2nx)$ and the above inequality,
\begin{align*}
\frac{u}{w} &=\frac {\phi^{-1}(1/x)}{\phi^{-1}(1/y)} \le \frac
{\phi^{-1}(1/x)}{\phi^{-1}(1/(2nx))} \le K  \frac
{\phi^{-1}(1/x)}{\phi^{-1}(1/(nx))} \\
& = K\left\| \sum_{i=1}^n f_i
\right\| \le C_p \left(\sum_{i=1}^n \|f_i\|^p\right)^{\frac1p}  =  KC_p n^{\frac1p} \le KC_p
\left(\frac{y}{x}\right)^{\frac1p} = KC_p \left(\frac{\phi(u)}{\phi(w)}\right)^{\frac1p}.
\end{align*}
 This implies that for every $u>w\ge
1$, 
\[
\phi(u) \ge (KC_p)^{-p} ({u}/{w})^p \phi(w),
\]
and  thus setting $s=w$ and $as=u$ where $a\ge 1$ and $s\ge 1$, 
\[
\phi(as)\ge (KC_p)^{-p} a^p\phi(s),
\]
which means $\Delta^{*p}$ condition of $\phi$ for large arguments. 

{\it Proof of \rm(iii) $\to$ \rm(i)} {\it if $\phi$ is convex.}

In this case $L_\phi$ is a Banach space and we can use general duality between convexity and concavity. 
In view of Proposition \ref{prop:dual},  $\phi^*$ satisfies condition $\Delta^{p'}$. Thus by Theorem \ref{th:concavity},
$L_{\phi^*}$ must be $p'$-concave.  Moreover $\phi^*\in \Delta^{p'}$ implies that $\phi^*\in\Delta_2$ in view of Proposition \ref{prop:delta2}. Therefore $L_{\phi^*}= E_{\phi^*}$, and so $(L_{\phi^*})^* \simeq L_\phi$. Consequently, in view of Theorem \ref{th:LT1f18}, 
 $ L_\phi$ is $p$-convex.

{\it Proof of \rm(iii) $\to$ \rm(i)} {\it for arbitrary Orlicz function $\phi$.}

 We can give a proof analogous to the last part of the proof in  Theorem \ref{th:concavity} replacing concavity by convexity. However we can treat this more generally. Recall, given $0< p < \infty$ and a quasi-Banach lattice $(X, \|\cdot\|_X)$,  the $p$-convexification of $X$ is the space  $X^{(p)} = \{x\in X : |x|^p \in X\}$ equipped with  the quasi-norm $\|x\|_{X^{(p)}} = \|\,|x|^p\,\|_X^{1/p}$. It is straightforward to show that  $X^{(p)}$  is $p$-convex \cite[page 53]{LT2}. 

Now if $\phi$ satisfies condition $\Delta^{*p}$, then in
view of Proposition \ref{prop:delta*p}, $\phi$ is equivalent to an Orlicz function
$\psi$ such that $\psi(u^{1/p})$ is convex. Thus $L_{\psi(
u^{1/p})}$ is a Banach space. Moreover, $L_{\psi}= (L_{\psi(u^{1/p})})^{(p)}$ is a $p$-convexification of $L_{\psi(u^{1/p})}$. 
 Hence $L_\psi$ and so $L_{\phi}$ is $p$-convex.

\end{proof}

\section{type and cotype}

\begin{theorem}  \label{th:cotype}
Let $q\ge 2$  and $\phi$ be an Orlicz function. Then the Orlicz space $L_\phi$ has
cotype $q$  if and only if $\phi$ satisfies condition $\Delta^q$.
\end{theorem}

\begin{proof} Under the assumption of convexity on $\phi$, the space $L_\phi$ is a Banach lattice. Therefore we can use general facts from Banach spaces and lattices. The summary of the relations among several local geometric properties is given in two diagrams in the monograph \cite{LT2}.
By them, for $q>2$,  lower $q$-estimate is equivalent to cotype $q$. Moreover, $2$-concavity and cotype $2$ are also equivalent. 
Thus   in view of  Theorem \ref{th:concavity} the proof is done.

Assume now that $\phi$ is an arbitrary  Orlicz function, not necessarily convex,  and that $L_\phi$ has cotype $q<\infty$. We will show that $\phi$ satisfies condition $\Delta^q$. 

Observe  at first that
$\phi$ must satisfy condition $\Delta_2$.  Indeed,  in the opposite
case in view of Proposition \ref{prop:linfty}, $L_\phi$ contains an isomorphic copy of $l_\infty$. The space $l_\infty$ does not have any finite cotype. In fact let $e_i =(0,\dots,0,1,0,\dots)$, $i\in\mathbb{N}$, with $1$ on the $i$-th place  be unit vectors in $l_\infty$. Then 
\begin{equation}\label{eq:112}
\int_0^1\left\|\sum_{i=1}^n r_i(t)e_i\right\|_\infty\, dt = \left\|\sum_{i=1}^n e_i\right\|_\infty = 1 \ \ \ \ \text{and} \ \ \ \sum_{i=1}^n \|e_i\|_\infty^q = n,
\end{equation}
for any $n\in\mathbb{N}$ and $q>0$, which implies that $l_\infty$ does not have a finite cotype. 
It follows that $L_\phi$ does not have it either, which contradicts the assumption. 

We provide a further proof only in the case of finite non-atomic measure.  Assume that $\phi$ does not satisfy condition $\Delta^q$ for large arguments.   We employ now the functions $\{g_{in}\}_{i\in N_n}$ constructed in the proof of   Theorem \ref{th:concavity}. They  satisfy the inequalities (\ref{eq:1333}) and (\ref{eq:1444}).
 Hence for all $\lambda > 0$,
$I_\phi (\lambda \sum_{i\in N_n} g_{in}) \rightarrow 0$, equivalently that $\|\sum_{i\in N_n} g_{in} \| \rightarrow 0$ as
$n\rightarrow \infty$. By this, (\ref{eq:1333}) and the assumption that $L_\phi$ has cotype $q$,
\[
\frac12\le\sum_{i\in N_n} \|g_{in}\|^q \le K^q\left( \int_0^1\left\|\sum_{i\in N_n} r_i(t) g_{in}\right\| \, dt\right)^q = K^q \left\|\sum_{i\in N_n} g_{in} \right\|^q \to 0,
\]
as $n\to\infty$, which concludes that  $L_\phi$ has no cotype $q$, and thus $\phi\in \Delta^q$.

 Let now $\phi$ satisfy condition $\Delta^q$. Then by Theorem \ref{th:concavity}, $L_\phi$ is $q$-concave.

 By  the general assumption $\alpha_\phi>0$ and Remark \ref{rem:delta}, $\phi\in\Delta^{*r}$ for some $r>0$.  This  yields that $L_\phi$ is $r$-convex by Theorem \ref{th:convexity}.
Therefore in view of the generalization of   Khintchine's inequality given by Theorem \ref{th:Khinchine}  and $2\le q<\infty$, 
\[
\int_0^1\left\|\sum_{i=1}^n r(t) f_i\right\|_\phi\,dt \ge C \left\|\left(\sum_{i=1}^n |f_i|^2\right)^{1/2}\right\|_\phi \ge C \left\|\left(\sum_{i=1}^n |f_i|^q\right)^{1/q}\right\|_\phi \ge CD_q\left(\sum_{i=1}^n \|f_i\|_\phi^q\right)^{1/q},
\]
where $D_q$ is a $q$-concavity constant of $L_\phi$.
Hence $L_\phi$ has cotype $q$.

\end{proof}

\begin{proposition}\label{prop:type1} Let $0<p\le 2$. If $L_\phi$ has type $p$ then $\phi$ satisfies condition $\Delta^{*p}$.
\end{proposition}
\begin{proof}
We provide a proof only in the case of finite non-atomic measure. Assume without loss of generality that $\phi(1) = 1$ and $\mu(\Omega) = 1$. 

Assume that $\phi$ does not satisfy condition $\Delta^{*p}$ for large arguments. There
exist infinite sequences $\{a_n\}$ and $\{u_n\}$, such that $1\le
u_n \rightarrow \infty$ as $n\rightarrow \infty$, $a_n\ge 1$ and
\begin{equation}\label{eq:16}
\phi(a_n^{1/p}u_n) \le a^{-n}\, a_n\,\phi(u_n),
\end{equation}
for every $n\in\Bbb N$ and some $a  \ge 2$. Observe that the sequence 
 $\{a_n\}$ is not bounded.
  If for a contrary, for some $K>0$, $a_n\le K$ for all $n\in\mathbb{N}$, then  
\[
\phi(u_n)\le \phi(a_n^{1/p} u_n)\le a^{-n} a_n \phi(u_n) \ \ \text{so}\ \ \ 1 \le \frac{a_n}{a^n} \le \frac{K}{a^n} \to 0 \ \ \text{as} \ \ \  n\to\infty,
\]
 which is a contradiction.  Recall, given $a\in \mathbb{R}$, the ceiling function $\ceil*{a}$ is the smallest integer bigger than or equal to 
$a$.  Let $\ceil*{a_0} = 0$ and
\[
N_n = \left\{ 1 + \sum_{i=0}^{n-1} \ceil*{a_i},  2 + \sum_{i=0}^{n-1} \ceil*{a_i},
\dots, \ceil*{a_n} + \sum_{i=0}^{n-1} \ceil*{a_i} \right\},
\]
for every $n\in \Bbb N$.  The sequence $\{N_n\}$ is a disjoint
partition of  natural numbers $\Bbb N$, and each $N_n$ contains
exactly $\ceil*{a_n}$ consecutive natural numbers.   
By the assumption $\alpha_\phi > 0$, in view of Remark \ref{rem:delta}, $\phi\in\Delta^{*s}$ for some $p>s>0$. Further by Proposition \ref{prop:delta*p}, 
 the function $\phi(u^{1/s})$ is equivalent to a convex function. Assume thus not loosing generality that $\phi(u^{1/s})$ is  convex.  Thus by (\ref{eq:16}),
\begin{equation}\label{eq:166}
\ceil*{a_n} \phi(u_n) \ge 2^n \phi(a_n^{1/p} u_n) = 2^n \phi((a_n^{s/p} u_n^s)^{1/s}) \ge \phi(2^{n/s} a_n^{1/p} u_n).
\end{equation}
Clearly
\[
\frac{\ceil*{a_n}}{\phi(2^{n/s} a_n ^{1/p} u_n)} \le \frac{a_n+1}{\phi(2^{n/s} a_n ^{1/p})}.
\]
Consider now the convex function 
$\psi(u) = \phi(u^{1/s})$.
It follows that $\psi(2^x x^{s/p}) \ge 2^x x^s \psi(1)$ for $x\ge 1$.  Therefore 
\[
\lim_{x\to\infty}\frac{x+1}{\phi(2^{x/s} x^{1/p})} =  \lim_{x\to\infty}\frac{x+1}{\psi(2^x x^{s/p})} \le \lim_{x\to\infty}\frac{2x}{2^x x^{s/p} \psi(1)} = \lim_{x\to\infty}\frac{2x^{1-\frac{s}{p}}}{2^x \psi(1)} =0.
\]
Hence choosing a subsequence if necessary we can assume that
\[
\sum_{n=1}^\infty \frac{\ceil*{a_n}}{\phi(2^{n/s} a_n ^{1/p} u_n)} < \mu(\Omega) = 1.
\] 
Therefore for every $n\in\Bbb
N$ there exists a finite sequence of disjoint sets $\{A_{in}\}_{i\in N_n}\subset \Omega$
such that 
$$
\mu (A_{in}) = \frac1{\phi(2^{n/s}a_n^{1/p}u_n)}
$$
for each $i\in N_n$, $n\in \Bbb N$.  Defining for $i\in N_n$ and $n\in \Bbb N$,
\[
g_{in} = u_n\chi_{A_{in}},
\]
 it follows  that  
\[
I_\phi (2^{n/s} a_n^{1/p}g_{in}) = \phi (2^{n/s} a_n^{1/p}u_n) \mu
(A_{in}) = 1.
\]
Consequently,  
\begin{equation}\label{eq:13}
\| g_{in} \|^p = \frac{1}{2^{np/s} a_n} \ \ \ \text{and}\ \  \ \sum_{i\in N_n} \|g_{in}\|^p =  \frac{\ceil*{a_n}}{2^{np/s}a_n} \le  \frac
{2}{2^{np/s}} \to 0,
\end{equation}
as $n\to \infty$.
On the other hand 
 by (\ref{eq:166}), for every $n\in\mathbb{N}$,
\[
I_\phi\left( \sum_{i\in N_n} g_{in}\right) =
\ceil*{a_n}\phi(u_n)\mu(A_{in}) = \frac{\ceil*{a_n}\phi(u_n)}{ \phi (2^{n/s}a_n^{1/p}u_n)}  \ge 1.
\]
Now since $g_{in}$ are disjointly supported for $i\in N_n$, we get 
\begin{equation}\label{eq:14}
K\left(\sum_{i\in N_n} \|g_{in}\|^p\right)^{1/p}\ge \int_0^1\left\|\sum_{i\in N_n} r_i(t) g_{in}\right\| \, dt= \left\|\sum_{i\in N_n} g_{in}\right\| \ge 1. 
\end{equation}
Finally combining (\ref{eq:13}) and (\ref{eq:14}),  $L_\phi$ cannot have type $p$, which which finishes the proof.

\end{proof}

\begin{proposition}\label{prop:type2}
Let $1<p\le 2$. If $L_\phi$ has type $p$ then $\phi$ satisfies condition $\Delta_2$.

\end{proposition}

\begin{proof}

If   $\phi$ does not satisfy $\Delta_2$ then  by Proposition \ref{prop:linfty}, $L_\phi$ contains an isomorphic copy of $l_\infty$, which by reasoning as in (\ref{eq:112}), does not have any finite cotype.
Thus $L_\phi$ does not posses any finite cotype either.  Now observe that $L_\phi$ can be treated as Banach space. Indeed if $L_\phi$ has type $p>1$ then by Theorem \ref{th:kal1}, the space $L_\phi$ is normable.   Consequently in view of Theorem \ref{th:1f13} it cannot have  type bigger than $1$. 

\end{proof}

\begin{theorem} \label{th:typecriterion}
Given $1<p\le 2$ and an Orlicz function $\phi$, the Orlicz space $L_\phi$ has
type $p$ if and only if  $\phi$ satisfies both conditions $\Delta_2$ and $\Delta^{*p}$.
\end{theorem}
\begin{proof}

{\it Proof in the case $1<p\le2$ and a convex $\phi$.}

Here we use general relations in Banach spaces between type and convexity, and duality between type and cotype.
 
Let  $\phi$ satisfy $\Delta_2$ and $\Delta^{*p}$. 
 It follows that $L_\phi$ is $p$-convex by Theorem \ref{th:convexity}. Moreover, 
by Proposition \ref{prop:delta2},
$\phi\in\Delta^q$ for some  $1<q<\infty$, and hence $L_\phi$ is $q$-concave in view of Theorem \ref{th:concavity}.
 Applying now the diagram on page 101 in \cite{LT2}, $L_\phi$ has type $p$.

If $L_\phi$ has type $p>1$, then due to Theorem \ref{th:1f13}, the
space is $q$-concave for some $q<\infty$. Consequently,  it cannot contain an order isomorphic copy of $l_\infty$ and so by Proposition \ref{prop:linfty}, $\phi$ must satisfy condition
$\Delta_2$. It  follows that  $(L_\phi)^* \simeq L_{\phi^*}$. This in turn implies that $L_{\phi^*}$ has cotype $p'$ from Theorem \ref{th:LT1f18}.  Now by Theorem \ref{th:cotype}, $\phi^*\in\Delta^{p'}$ and then in view of Proposition \ref{prop:dual}, $\phi\in\Delta^{*p}$.

\vspace{2mm}

{\it Proof in  case $1< p\le2$ and an arbitrary Orlicz function $\phi$.}

The {\it necessity} follows from Propositions \ref{prop:type1} and \ref{prop:type2}.

{\it Sufficiency.}  Let now $\phi\in \Delta_2$ and $\phi\in\Delta^{*p}$. 

By Theorem \ref{th:convexity}, the condition $\Delta^{*p}$ implies that the space $L_\phi$ is $p$-convex.  

 In view of Proposition \ref{prop:delta2} the assumption $\Delta_2$ for $\phi$ implies that $\phi\in \Delta^q$ for some $q<\infty$, and so $L_\phi$ is $q$-concave by Theorem \ref{th:concavity}.  Then applying Theorem \ref{th:Khinchine} for $L_\phi$ and  $0<p\le 2$, for any $f_1,\dots,f_n \in L_\phi$ we get
\[
\int_0^1\left\|\sum_{i=1}^n r(t) f_i\right\|_\phi\,dt \le C \left\|\left(\sum_{i=1}^n |f_i|^2\right)^{1/2}\right\|_\phi \le C \left\|\left(\sum_{i=1}^n |f_i|^p\right)^{1/p}\right\|_\phi \le C \left(\sum_{i=1}^n \|f_i\|_\phi^p\right)^{1/p}.
\]
Hence $L_\phi$ has type $p$.

\end{proof}

\begin{remark}Observe that the proof of sufficiency in Theorem \ref{th:typecriterion} works for any  $0<p\le 2$. 
\end{remark}

\begin{remark} Kalton \cite[Theorem 4.2]{Kal1981} showed that for $0<p<1$, any quasi-Banach space is $p$-normable if and only if it has type $p$.  For $p=1$ this is no longer true.  Clearly any normable space has type $1$, however there exist quasi-Banach spaces that posses type $1$ but they are not normable \cite{Kal1981}. Natural examples of such spaces are weak-$L^1$ spaces considered in \cite{KalKam, Kam2018}.  In this respect Orlicz spaces  behave more regularly. As we see in the  next corollaries, normability is equivalent to type $1$ in $L_\phi$.

\end{remark}


\begin{corollary}\label{cor:pnormability}
Let $\phi$ be an Orlicz function and $0<p\le 1$. Then $\rm(i), (ii)$ and $\rm(iii)$ are equivalent.
\begin{itemize}
\item[(i)] $\phi$ satisfies condition $\Delta^{*p}$.
\item[(ii)]  $L_\phi$ is $p$-normable.
\item[(iii)]   $L_\phi$ has type $p$.
\end{itemize}
\end{corollary}
\begin{proof}
(i) $\Rightarrow$ (ii) By Theorem \ref{th:convexity}, $L_\phi$ is $p$-convex. Then by the inequality $|a+b|^p \le |a|^p + |b|^p$, $a,b\in \mathbb{R}$, we get
\[
\left\|\sum_{i=1}^n f_i\right\| \le \left\|\left(\sum_{i=1}^n |f_i|^p\right)^{1/p}\right\| \le C_p\left( \sum_{i=1}^n \|f_i\|^p\right)^{1/p}
\]
for  any $f_1,\dots,f_n\in L_\phi$, $n\in\mathbb{N}$. Thus $\|\cdot\|$ is $p$-normable. 

(ii) $\Rightarrow$ (iii)
Let $f_1,\dots,f_n\in L_\phi$.  By $p$-normability, there exists $C>0$ such that
\[
\int_0^1 \left\|\sum_{i=1}^n r_i(t) f_i \right\| dt \le \left\|\sum_{i=1}^n |f_i|\right\| \le C \left(\sum_{i=1}^n \|f_i\|^p\right)^{1/p}.
\]

(iii) $\Rightarrow $ (i) See Proposition \ref{prop:type1}.

\end{proof}
 In view of Proposition \ref{prop:delta*p}, the immediate consequence of Corollary \ref{cor:pnormability} for $p=1$ is the following result. 
 
\begin{corollary}\label{cor:1normability}
Let $\phi$ be an Orlicz function. Then $\rm(i), (ii)$ and $\rm(iii)$ are equivalent.
\begin{itemize}
\item[(i)] $\phi$ is equivalent to a convex function. 
\item[(ii)]  $L_\phi$ is normable.
\item[(iii)]   $L_\phi$ has type $1$.
\end{itemize}
\end{corollary}

\section{Concluding results}

Summarizing the results from two previous sections, Propositions \ref{prop:deltap}, \ref{prop:delta*p},  Theorems \ref{th:concavity}, \ref{th:convexity}, \ref{th:cotype}, \ref{th:typecriterion} and Corollary \ref{cor:pnormability}, we obtain the next two theorems.

\begin{theorem}  
Let $\alpha_\phi> 0$ and $0< q < \infty$.  Consider the following properties.
\begin{itemize}
\item[{\rm (i)}] $\phi$ satisfies condition $\Delta^{q}$.
\item[{\rm (ii)}] There exists an Orlicz function $\psi$
equivalent to
$\phi$ such that $\psi(u^{\frac{1}{q}})$ is concave.
\item[{\rm (iii)}] $L_\phi$ is $q$-convex.
\item[{\rm (iv)}] $L_\phi$ satisfies a lower
$q$-estimate.
\item[{\rm (v)}] $L_\phi$ has cotype  $q$.
\end{itemize}
The conditions  {\rm (i)-(iv)} are equivalent. If $q\ge 2$, then all conditions {\rm (i)-(v)} are  equivalent.

\end{theorem}

\begin{theorem}
Let $\alpha_\phi > 0$ and $0<p < \infty$. Consider the conditions below.
\begin{itemize}
\item[{\rm (i)}] $\phi$ satisfies condition $\Delta^{*p}$.
\item[{\rm (ii)}] There exists an Orlicz function $\psi$
equivalent to
$\phi$ such that $\psi(u^{\frac{1}{p}})$ is convex.
\item[{\rm(iii)}] 
$L_\phi$ satisfies an
upper $p$-estimate.
\item[{\rm (iv)}] $L_\phi$ is $p$-convex.
\item[{\rm (v)}] $L_\phi$ is  $p$-normable.
\item[{\rm (vi)}] $\phi$ satisfies conditions $\Delta_2$ and $\Delta^{*p}$.
\item[{\rm (vii)}] $L_\phi$ has type  $p$.
\end{itemize}

If $0<p<\infty$ then {\rm (i)-(iv)} are equivalent.

If $1< p \le 2$ then  {\rm (vi)-(vii)} are equivalent.

If $0<p\le 1$, then  {\rm (i)-(v)} and {\rm (vii)} are equivalent.

\end{theorem}

For Banach spaces the notions of type and cotype are related to geometric properties of  subspaces uniformly isomorphic to finite dimensional spaces $l_1^n$ and $l_\infty^n$, $n\in\mathbb{N}$, respectively. Recall that $l_1^n$ and $l_\infty^n$ are $n$-dimensional spaces equipped with $l_1$- and $l_\infty$-norm respectively. A Banach space has always type $1\le p\le 2$. If it has type $p=1$, then we say that it has a trivial type. 

\begin{corollary} Let $\phi$ be a convex function. The following conditions are equivalent.
\begin{itemize}
\item[$(1)$] $L_\phi$ has nontrivial type.
\item[$(2)$] $L_\phi$ does not contain ${l_1^n}'s$ uniformly.
\item[$(3)$]  $L_\phi$ is $B$-convex.
\item[$(4)$]  $1<\alpha_\phi \le \beta_\phi< \infty$, that is $\phi$ and
$\phi^*$ satisfy condition $\Delta_2$.
\item[$(5)$] $L_\phi$ is reflexive. 
\end{itemize}
\end{corollary}
\begin{proof}
By the assumption of convexity of $\phi$, $L_\phi$ is a Banach space. The equivalence of (1) and (2) holds true for any Banach space by Pisier's Theorem \cite[Theorems 13.3]{DJT}, as well as the equivalence of (2) and (3) by \cite[Theorem 13.6]{DJT}. 
By Theorem \ref{th:typecriterion}, $L_\phi$ has type $1<p\le 2$ if and only if $\phi\in\Delta^{*p}$ and $\phi\in\Delta_2$. By Proposition \ref{prop:delta2}, the latter condition is equivalent to $\beta_\phi<\infty$, while the previous one is equivalent to $\alpha_\phi>1$.
Moreover $L_\phi$ is reflexive if and only if $\phi$ and $\phi^*$ satisfy condition $\Delta_2$ \cite{BS, KR, Lux}, so (4) is equivalent to (5).

\end{proof}

\begin{corollary} Let $\phi$ be a convex function. All conditions below are equivalent.
\begin{itemize}
\item[$(1)$]  $L_\phi$ has a finite cotype.
\item[$(2)$]  $L_\phi$ does not contain ${l_\infty^n}'s$ uniformly.
\item[$(3)$]  $\beta_\phi< \infty$, that is $\phi$ satisfies condition
$\Delta_2$.
\end{itemize}
\end{corollary}

\begin{proof} For any Banach space, condtions (1) and (2) are equivalent \cite[Theorem 14.1]{DJT}. Moreover by Theorem \ref{th:cotype}, $L_\phi$ has cotype $2\le q < \infty$ if and only if $\phi\in \Delta^q$.  In view of Proposition  \ref{prop:delta2}, the existence of such a $q$, in turn, is equivalent to $\beta_\phi< \infty$.

\end{proof}

\begin{corollary} Let $\phi$ be an Orlicz function with $\alpha_\phi > 0$.

 {\rm{\bf I}}. The Orlicz space $L_\phi$ has cotype
$\max (\beta_\phi + \varepsilon, 2)$ for every $\varepsilon > 0$.
If
$L_\phi$ has a finite cotype $q$ then $q\ge \max (\beta_\phi , 2)$.

 {\rm{\bf II}}. The Orlicz space $L_\phi$ has type $\min (\alpha_\phi -
\varepsilon , 2)$, for every  $\varepsilon > 0$ with $\alpha_\phi -
\varepsilon  > 0$,  whenever
$ \beta_\phi < \infty$.  If $L_\phi$ has a
nontrivial type $p>0$, then $\beta_\phi < \infty$ and $p\le \min
(\alpha_\phi, 2)$.
\end{corollary}

\begin{proof} Part {\bf I} follows from Remark \ref{rem:delta} and Theorem \ref{th:cotype}.
Part {\bf II} is a result of Remark \ref{rem:delta} and Theorem \ref{th:typecriterion}.

\end{proof}

\begin{examples}
\par (1) For any $p\in (0,2]$ and $q\in [2,\infty)$, there exists an
Orlicz space of type $p$ and cotype $q$.

(a) Let $L_\phi$ be induced by $\phi(u) = \max (u^p , u^q)$, $0<p\le 2\le q<\infty$. Then $\phi$ is an Orlicz function not necessarily convex with $\alpha_\phi = p$ and $\beta_\phi = q$. If $1\le p\le 2$ then $\phi$ is convex and $L_\phi$ is a Banach space.

Moreover $L_\phi =L_p\cap L_q$ where $\|\cdot\|_\phi$ is a quasi-norm equivalent to $\max\{\|\cdot\|_p, \|\cdot\|_q\}$. Here $\|f\|_r = \left(\int |f|^r\right)^{1/r}$,  $r>0$. It is easy to check that $\phi\in\Delta_2$ and $\phi\in\Delta^{*p}$, and $\phi\in\Delta^q$. So by Theorem \ref{th:typecriterion} it has type $p=\alpha_\phi$, and by Theorem \ref{th:cotype} it has cotype $q=\beta_\phi$.

(b) Let $L_\phi = L_p +
L_q $, where $\phi(u)= \min (u^p,
u^q)$, $0<p\le 2\le q<\infty$. Then $\alpha_\phi = p$ and $\beta_\phi = q$, and $\|f\|_\phi$ is equivalent to $ \inf\{\|g\|_p + \|h\|_q: f= g+h,\ g\in L_p, \ h\in L_q\}$. 
Again as above $\phi\in\Delta_2$ and $\phi\in\Delta^{*p}$, and $\phi\in\Delta^q$. Thus $L_\phi$ has type $p=\alpha_\phi$ and cotype $q=\beta_\phi$.

The function $\phi(u) = \min(u^p,u^q)$ is never convex. However if $1\le p \le 2$ then it is equivalent to a convex Orlicz function.
In fact $\phi$ is equivalent to 
\[
\Phi(u) =\int_0^u \min (t^{p-1},
t^{q-1})dt = \int_0^u \frac{\phi(t)}{t} dt.
\] 
Obviously $\frac{\phi(u)}{u}$ is increasing for $u>0$, so $\Phi$ is convex. Moreover, for every $u>0$,
\[
\phi\left(\frac{u}{2}\right) \le \Phi(u) \le \phi(u).
\]
In this case the quasinorm $\|\cdot\|_\phi$ is equivalent to the Luxemburg norm $\|\cdot\|_\Phi$, that is $L_\phi$ is normable.

If $F:\Bbb R \rightarrow \Bbb R$ is any function then
$F^*(v) = \sup_{u\in\Bbb R}\{uv - F(u)\}$, $v\in\mathbb{R}$,  is said to be the {\it
Legendre transform} $F$.
It is well known that $F = F^{**}$ if and only if $F$ is convex \cite[Theorem 1 on p. 175]{IT}. Since $F^{**}\le F$ for any $F$, $F^{**}$ is a
convex minorant of $F$. 

In our case $\phi^{**}\le \phi$ is a convex minorant of $\phi$. Moreover $\Phi^{**} = \Phi$, and thus 
\[
\phi^{**}\left(\frac{u}{2}\right) \le \Phi(u) \le \phi^{**}(u).
\]
Therefore $\phi$ is equivalent to its convex minorant.

\vspace{2 mm}

\par (2) An Orlicz space need not have its upper index as a cotype. Let
$q\ge 2$ and define
\[
\phi(u) =
\begin{cases}
0, &\text{for $u = 0$}\\
\frac{u^q}{|\ln u |}, &\text{for $0<u\le \frac1{e}$}\\
(\frac1q + 1)u^q - \frac1q e^{-q}, &\text{for
$u>\frac1{e}$.}
\end{cases}
\]
Then $\beta_\phi = q$, but $\sup_{\lambda \ge 1 ,
u>0}\frac{\phi(\lambda u)}{\lambda^q\phi(u)} \ge \sup_{\lambda \ge 1} \frac{|\ln (1/e) |}{|\ln (\lambda /e)|} 
 = \infty$, which implies that $\phi$ does not satisfy condition $\Delta^q$. Thus
$L_\phi$ has cotype $q + \varepsilon$ for every $\varepsilon > 0$,
but does not have cotype $q$.
\end{examples}

\end{document}